\documentclass[reqno]{amsart}

\input xy
\xyoption{all}
\usepackage{epsfig}
\usepackage{color}
\usepackage{amsthm}
\usepackage{amssymb}
\usepackage{amsmath}
\usepackage{amscd}
\usepackage{amsopn}
\usepackage{url}
\usepackage{hyperref}\hypersetup{colorlinks}

\usepackage{color} %\textcolor{red}{Test}

\definecolor{darkred}{rgb}{1,0,0} %can change the intensity in [0,1]
\definecolor{darkgreen}{rgb}{0,0.8,0}
\definecolor{darkblue}{rgb}{0,0,1}

\hypersetup{colorlinks,
linkcolor=darkblue,
filecolor=darkgreen,
urlcolor=darkred,
citecolor=darkgreen}

\newcommand{\labell}[1] {\label{#1}}
\numberwithin{equation}{section}
\newtheorem {Theorem}{Theorem}
\numberwithin{Theorem}{section}

\newtheorem {Lemma}[Theorem]    {Lemma}

\theoremstyle{definition}

\theoremstyle{remark}
\newtheorem{Remark}[Theorem]{Remark}
\newtheorem{Example}[Theorem]{Example}

\expandafter\chardef\csname pre amssym.def at\endcsname=\the\catcode`\@
\catcode`\@=11
\def\undefine#1{\let#1\undefined}
\def\newsymbol#1#2#3#4#5{\let\next@\relax
 \ifnum#2=\@ne\let\next@\msafam@\else
 \ifnum#2=\tw@\let\next@\msbfam@\fi\fi
 \mathchardef#1="#3\next@#4#5}
\def\mathhexbox@#1#2#3{\relax
 \ifmmode\mathpalette{}{\m@th\mathchar"#1#2#3}%
 \else\leavevmode\hbox{$\m@th\mathchar"#1#2#3$}\fi}
\def\hexnumber@#1{\ifcase#1 0\or 1\or 2\or 3\or 4\or 5\or 6\or 7\or 8\or
 9\or A\or B\or C\or D\or E\or F\fi}

\font\teneufm=eufm10
\font\seveneufm=eufm7
\font\fiveeufm=eufm5
\newfam\eufmfam
\textfont\eufmfam=\teneufm
\scriptfont\eufmfam=\seveneufm
\scriptscriptfont\eufmfam=\fiveeufm

\catcode`\@=\csname pre amssym.def at\endcsname

\def    \eps    {\epsilon}

\newcommand{\CA}{{\mathcal A}}

\newcommand{\CM}{{\mathcal M}}

\newcommand{\CS}{{\mathcal S}}

\newcommand{\id}{{\mathit id}}
\newcommand{\pt}{{\mathit pt}}
\newcommand{\const}{{\mathit const}}

\newcommand{\ta}{\tilde{a}}

\newcommand{\ka}{{\kappa}}

\newcommand{\tH}{\tilde{H}}

\newcommand{\tA}{\tilde{\mathcal A}}

\newcommand{\PP}{{\mathcal P}}
\newcommand{\bPP}{\bar{\mathcal P}}

\def    \nat    {{\natural}}
\def    \F      {{\mathbb F}}

\def    \C      {{\mathbb C}}
\def    \R      {{\mathbb R}}

\def    \Z      {{\mathbb Z}}

\def    \T      {{\mathbb T}}
\def    \CP     {{\mathbb C}{\mathbb P}}

\def    \12    {{\frac{1}{2}}}

\def    \p      {\partial}
\def    \codim  {\operatorname{codim}}

\def    \Area  {\operatorname{Area}}

\def    \HF     {\operatorname{HF}}

\def    \HQ     {\operatorname{HQ}}

\def    \Gr     {\operatorname{Gr}}
\def    \H     {\operatorname{H}}

\def    \CF     {\operatorname{CF}}

\def    \bPP     {\bar{\mathcal{P}}}
\def    \bx     {\bar{x}}
\def    \by     {\bar{y}}

\def    \bg    {\bar{\gamma}}

\def    \MUCZ  {\operatorname{\mu_{\scriptscriptstyle{CZ}}}}

\begin{document}

%%%%%%%%%%%%%%%%%%%%%%%%%%%%%%
%   TEXT FORMATTING

\setlength{\smallskipamount}{6pt}
\setlength{\medskipamount}{10pt}
\setlength{\bigskipamount}{16pt}

%%%%%%%%%%%%%%%%%%%%%%%%%%

%%%%%%%%%%%%%%%%%%%%%%%%%%

%%%%%%%%%%%           BEGINNING OF  TEXT

%%%%%%%%%%%%%%%%%%%%%%%%%%

\title[Hyperbolic Fixed Points]{Hyperbolic Fixed Points and Periodic
  Orbits of Hamiltonian Diffeomorphisms}

\author[Viktor Ginzburg]{Viktor L. Ginzburg}
\author[Ba\c sak G\"urel]{Ba\c sak Z. G\"urel}

\address{BG: Department of Mathematics, Vanderbilt University,
  Nashville, TN 37240, USA} 
\email{basak.gurel@vanderbilt.edu}

\address{VG: Department of Mathematics, UC Santa Cruz, Santa Cruz, CA
  95064, USA} \email{ginzburg@ucsc.edu}


\date{\today} 

\thanks{The work is partially supported by NSF grants DMS-0906204 (BG)
  and DMS-1007149 (VG) and by the faculty research funds of the
  University of California, Santa Cruz.}

\bigskip

\begin{abstract} We prove that for a certain class of closed monotone
  symplectic manifolds any Hamiltonian diffeomorphism with a
  hyperbolic fixed point must necessarily have infinitely many
  periodic orbits. Among the manifolds in this class are complex
  projective spaces, some Grassmannians, and also certain 
  product manifolds such as the product of a projective space with a
  symplectically aspherical manifold of low dimension. 

  A key to the proof of this theorem is the fact that the energy
  required for a Floer connecting trajectory to approach an iterated
  hyperbolic orbit and cross its fixed neighborhood is bounded away
  from zero by a constant independent of the order of iteration. This
  result, combined with certain properties of the quantum product
  specific to the above class of manifolds, implies the existence of
  infinitely many periodic orbits.
\end{abstract}

\maketitle

\tableofcontents

\newpage 

\section{Introduction and the main result}
\labell{sec:main-results}

\subsection{Introduction}
\label{sec:intro}
The main result of the paper is that for a certain class of closed
monotone symplectic manifolds any Hamiltonian diffeomorphism with a
hyperbolic fixed or periodic point must necessarily have infinitely
many periodic orbits. This class of manifolds includes complex
projective spaces, some Grassmannians, and also certain product
manifolds such as the product of a projective space with a
symplectically aspherical manifold of low dimension.

To put this result in perspective, recall that, for many closed
symplectic manifolds, every Hamiltonian diffeomorphism has infinitely
many periodic orbits and, in fact, periodic orbits of arbitrarily
large (prime) period whenever the fixed points are isolated. This
unconditional existence of infinitely many periodic orbits is often
referred to as the Conley conjecture. As of this writing, the Conley
conjecture has been shown to hold for all symplectic manifolds $M$
with $c_1(TM)\mid_{\pi_2(M)}=0$ and also for negative monotone
manifolds; see \cite{CGG,GG:gaps,He:irr} and also
\cite{FH,Gi:CC,GG:neg-mon, Hi,LeC,SZ} for some important, but less
general, results.  Ultimately, one can expect the Conley conjecture to
hold for most symplectic manifolds.

There are, however, notable exceptions. The simplest one is $S^2$: an
irrational rotation of $S^2$ about the $z$-axis has only two periodic
orbits, which are also the fixed points; these are the poles. In fact,
any manifold that admits a Hamiltonian torus action with isolated
fixed points also admits a Hamiltonian diffeomorphism with finitely
many periodic orbits. For instance, such a diffeomorphism is generated
by a generic element of the torus. In particular, complex projective spaces,
the Grassmannians, and, more generally, most of the coadjoint
orbits of compact Lie groups as well as symplectic toric
manifolds all admit Hamiltonian diffeomorphisms with finitely many
periodic orbits.

An analogue of the Conley conjecture applicable to such manifolds is
the conjecture that a Hamiltonian diffeomorphism with ``more than
necessary'' fixed points has infinitely many periodic orbits. Here
``more than necessary'' is left deliberately vague although, from the
authors' perspective, it should be interpreted as a lower bound
arising from some version of the Arnold conjecture. For $\CP^n$, the
expected threshold is $n+1$. This variant of the Conley conjecture is
inspired by a celebrated theorem of Franks asserting that a
Hamiltonian diffeomorphism of $S^2$ with at least three fixed points
must have infinitely many periodic orbits, \cite{Fr1,Fr2}; see also
\cite{FH,LeC} for further refinements and \cite{BH,CKRTZ,Ke:JMD} for
symplectic topological proofs. (It is worth pointing out that Franks'
theorem holds for area preserving homeomorphisms. However, the
discussion of possible generalizations of this stronger result to
higher dimensions is far outside the scope of this paper.)

There are few results directly supporting this conjecture in
dimensions greater than two. Some evidence is provided by the results
of \cite{Gu}, where a ``local version'' of the conjecture is
considered. The main theorem of the present paper can also be viewed
as supporting the conjecture. In fact, the conjecture, at least for
non-degenerate Hamiltonian diffeomorphisms of $\CP^n$ and some other
manifolds, would follow if we could replace a hyperbolic fixed point
by a non-elliptic fixed point in the main theorem. (Of course, our
main result implies the non-degenerate case of Franks' theorem, but in
general it appears to add little to what is already known in dimension
two.)

Hyperbolicity is central to the proof of the theorem. The argument
hinges on a result, perhaps of independent interest, asserting 
that the energy required for a Floer connecting trajectory of an
iterated Hamiltonian to approach a hyperbolic orbit and cross its
fixed neighborhood cannot be arbitrarily small: it is bounded away
from zero by a constant independent of the order of iteration.  This
is an exclusive feature of hyperbolic fixed points.

\subsection{Main theorem}
\label{sec:main-result}
Consider a closed monotone symplectic manifold $M^{2n}$ with minimal
Chern number $N$. We denote by $\HQ_*(M)$ the quantum homology ring of
$M$, for the sake of simplicity taken over $\F=\Z_2$, and by $q$ the
generator of the Novikov ring $\Lambda=\F[q,q^{-1}]$, normalized to
have degree $|q|=-2N$. We refer the reader to Section \ref{sec:prelim}
for a detailed discussion of our conventions and notation.

The main result of the paper is

\begin{Theorem}
\label{thm:main}
Let $M$ be a closed strictly monotone symplectic manifold of dimension $2n$
(i.e., $M$ is monotone and $c_1(TM)\mid_{\pi_2(M)}\neq 0$ and
$[\omega]\mid_{\pi_2(M)}\neq 0$) such that
$N\geq n/2+1$. Assume that 
\begin{equation}
\label{eq:main-hom-relation}
\beta *\alpha=q^\nu [M]
\end{equation}
in $\HQ_*(M)$ for some ordinary homology classes $\alpha\in \H_*(M)$
and $\beta\in \H_*(M)$ with $|\alpha|<2n$ and $|\beta|<2n$, and that
\begin{itemize}
\item[(i)] $\nu=1$ or
\item[(ii)] $|\alpha|\geq 3n+1-2N$.
\end{itemize}
Then any Hamiltonian diffeomorphism $\varphi_H$ of $M$ with a contractible
hyperbolic periodic orbit $\gamma$ has infinitely many periodic orbits. 
\end{Theorem}

\begin{Remark}[Non-contractible Orbits]
\label{rmk:non-contr}
The contractibility assumption on the orbit $\gamma$ can be dropped
once the monotonicity requirement is suitably adjusted. To be more
precise, assume that $M$ is toroidally monotone in addition to the
conditions of the theorem. (See Section \ref{sec:non-contr} for the
definition.) Then any Hamiltonian diffeomorphism of $M$ with a
hyperbolic periodic orbit $\gamma$ (not necessarily contractible) has
infinitely many periodic orbits. These orbits lie in the collection of
free homotopy classes formed by the iterations of $\gamma$. The proof
of this fact is essentially identical to the proof of the theorem; see
Remark \ref{rmk:non-contr-pf}.
\end{Remark}

Among the manifolds meeting the requirements of Theorem \ref{thm:main}
are complex projective spaces $\CP^n$, complex Grassmannians
$\Gr(2,N)$, $\Gr(3,6)$ and $\Gr(3,7)$, the products $\CP^m\times
P^{2d}$ and $\Gr(2,N)\times P^{2d}$ where $P$ is symplectically
aspherical and $d+2 \leq m$ in the former case and $d\leq 2$ in the
latter, and monotone products $\CP^m\times\CP^m$. These manifolds also
meet the requirements of the non-contractible version of the theorem
(Remark \ref{rmk:non-contr}), provided that the products are
toroidally monotone. (We do not have any example of a manifold with
$N\geq n/2+1$ satisfying the conditions of Case (ii) of the
theorem, but not of Case (i).)

\begin{Remark}
\label{rmk:period}
We emphasize that in Theorem \ref{thm:main} we do not make any
non-degeneracy assumptions on $\varphi_H$. Note also that, in contrast
with the Conley conjecture type results discussed above, we do not claim the
existence of periodic orbits with arbitrarily large period. 
\end{Remark}

\subsection{Discussion} As was mentioned in the introduction, a key to
the proof of Theorem \ref{thm:main} is a lower bound $c_\infty>0$,
independent of the order of iterations, on the energy that a Floer
trajectory of an iterated Hamiltonian requires to approach a
hyperbolic orbit and cross its fixed neighborhood. (The assumption
that the orbit is hyperbolic is essential here; see Remark
\ref{rmk:elliptic}.) This is Theorem \ref{thm:energy}, which is a
purely local result. Its proof relies on the Gromov compactness
theorem in the form established in~\cite{Fi}.

It is quite likely that the hypotheses of Theorem \ref{thm:main} can
be relaxed via, for instance, further refining Theorem
\ref{thm:energy}. (See also Remark \ref{rmk:refinement}.)
Hypothetically, a result similar to Theorem \ref{thm:main} should hold
for monotone manifolds without any restrictions on the minimal Chern
number $N$ or even for weakly monotone and rational symplectic
manifolds. Moreover, we expect the periodic orbits to enter an
arbitrarily small neighborhood of the hyperbolic orbit.  Note,
however, that the algebraic requirement \eqref{eq:main-hom-relation},
or some variant of it, is central to the proof.

In the context of Hamiltonian dynamical systems, the presence of one
hyperbolic orbit implies, $C^1$-generically, the existence of
transverse homoclinic points via the so-called connecting lemma; see
\cite{Ha,Xi}. (The genericity assumption is essential here, although
hypothetically this could be a $C^\infty$-generic condition rather
than $C^1$.) The existence of transverse homoclinic points has, in
turn, rich dynamical consequences among which is the
existence of infinitely many periodic orbits; see, e.g., \cite{Ru,Ze}
and references therein. Thus, under certain additional conditions on
the ambient manifold, Theorem \ref{thm:main} recovers a fraction of
this dynamics, but does this unconditionally rather than
generically. Note also that the existence of infinitely many periodic
orbits is a $C^1$-generic phenomenon, as follows from the closing
lemma (see \cite{PR}), and in many instances even $C^\infty$-generic
(see~\cite{GG:generic}).

Finally, note that one can expect an analogue of Theorem
\ref{thm:energy} to hold for hyperbolic periodic orbits of Reeb flows
and have applications, similar to Theorem \ref{thm:main}, to the
existence of infinitely many closed Reeb orbits; we will consider
this circle of questions elsewhere.

\subsection{Acknowledgements} The authors are grateful to Peter
Albers, Paul Biran, Anton Gorodetski, Helmut Hofer, Yng-Ing Lee, Yaron
Ostrover, Yasha Pesin, Dietmar Salamon, and Michael Usher for useful
discussions.  A part of this work was carried out while both of the
authors were visiting the IAS as a part of the Symplectic Dynamics
program and also during the first author's visits to the NCTS, NCKU,
Taiwan, and the FIM, ETH, Z\"urich. The authors would like to thank
these institutes for their warm hospitality and support.

\section{Preliminaries}
\label{sec:prelim}

The goal of this section is to set notation and conventions, following
mainly \cite{GG:gaps}, and to give a brief review of Floer
homology and several other notions used in the paper.

\subsection{Conventions and notation}
\label{sec:conventions}
Let $(M^{2n},\omega)$ be a closed symplectic manifold, which in this
paper (except Section \ref{sec:hyperbolic}) is assumed to be
\emph{monotone}, i.e., $[\omega]\mid_{\pi_2(M)}=\lambda
c_1(TM)\!\mid_{\pi_2(M)}$ for some $\lambda\geq 0$. Furthermore, we
focus on the case where $\lambda\neq 0$ and
$c_1(TM)\!\mid_{\pi_2(M)}\neq 0$ and refer to such manifolds as
\emph{strictly monotone}. (Note a somewhat unconventional position of
the \emph{monotonicity constant} $\lambda$ in the above definition.)
The positive generator $\lambda_0$ of the group
$\left<[\omega],{\pi_2(M)}\right>\subset\R$ formed by the integrals of
$\omega$ over the spheres is called the \emph{rationality
  constant}. Likewise, the \emph{minimal Chern number} $N$ is the positive
generator of the group
$\left<c_1(TM),{\pi_2(M)}\right>\subset\Z$. Clearly,
$\lambda_0=\lambda N$.

All Hamiltonians $H$ considered in this paper are assumed to be
$\ka$-periodic in time (i.e., $H$ is a function $S^1_\ka\times
M\to\R$, where $S^1_\ka=\R/\kappa\Z$) and the period $\ka$ is always a
positive integer.  When the period $\ka$ is not specified, it is equal
to one, and $S^1=\R/\Z$. We set $H_t = H(t,\cdot)$ for $t\in
S^1_\ka$. The Hamiltonian vector field $X_H$ of $H$ is defined by
$i_{X_H}\omega=-dH$. The (time-dependent) flow of $X_H$ will be
denoted by $\varphi_H^t$ and its time-one map by $\varphi_H$. Such
time-one maps are referred to as \emph{Hamiltonian diffeomorphisms}.
A one-periodic Hamiltonian $H$ can also be treated as
$\ka$-periodic. In this case, we will use the notation $H^{\nat \ka}$
and, abusing terminology, call $H^{\nat \ka}$ the $\ka$th iteration of
$H$.

Let $K$ and $H$ be one-periodic Hamiltonians such that $K_1=H_0$ and
$H_1=K_0$. We denote by $K\nat H$ the two-periodic Hamiltonian equal
to $K_t$ for $t\in [0,\,1]$ and $H_{t-1}$ for $t\in [1,\,2]$. Thus
$H^{\nat \ka}=H\nat \cdots\nat H$ ($\ka$ times).

Let $x\colon S^1_\ka\to M$ be a contractible loop. A \emph{capping} of
$x$ is a map $u\colon D^2\to M$ such that $u\mid_{S^1_\ka}=x$. Two
cappings $u$ and $v$ of $x$ are considered to be equivalent if the
integral of $c_1(TM)$ (and hence of $\omega$) over the sphere obtained
by attaching $u$ to $v$ is equal to zero. A capped closed curve
$\bar{x}$ is, by definition, a closed curve $x$ equipped with an
equivalence class of cappings. In what follows, the presence of
capping is always indicated by the bar.  We denote the collection of
capped one-periodic orbits of $H$ by $\bPP(H)$.

The action of a one-periodic Hamiltonian $H$ on a capped loop
$\bar{x}=(x,u)$ is defined by
$$
\CA_H(\bar{x})=-\int_u\omega+\int_{S^1} H_t(x(t))\,dt.
$$
The space of capped closed curves is a covering space of the space of
contractible loops, and the critical points of $\CA_H$ on this covering
space are exactly capped one-periodic orbits of $X_H$. The
\emph{action spectrum} $\CS(H)$ of $H$ is the set of critical values
of $\CA_H$. This is a zero measure set; see, e.g., \cite{HZ}. When $M$
is monotone as we are assuming here, $\CS(H)$ is closed, and hence
nowhere dense.  These definitions extend to $\ka$-periodic orbits and
Hamiltonians in the obvious way. Clearly, the action functional is
homogeneous with respect to iteration:
\begin{equation}
\label{eq:action-hom}
\CA_{H^{\nat \ka}}(\bx^\ka)=\ka\CA_H(\bx).
\end{equation}
Here $\bx^\ka$ stands for the $\ka$th iteration of the capped orbit $\bx$.

Throughout most of the paper, \emph{a periodic orbit is assumed
  to be contractible, unless explicitly stated otherwise}. More
specifically, we consider non-contractible orbits in Remarks
\ref{rmk:non-contr} and \ref{rmk:non-contr-pf}, and in Sections
\ref{sec:non-contr} and \ref{sec:hyperbolic}.

A periodic orbit $x$ of $H$ is said to be \emph{non-degenerate} if the
linearized return map $d\varphi_H \colon T_{x(0)}M\to T_{x(0)}M$ has
no eigenvalues equal to one, and $x$ is called \emph{hyperbolic} if
none of the eigenvalues has absolute value one.  A Hamiltonian is
non-degenerate if all its one-periodic orbits are non-degenerate.

Let $\bar{x}$ be a non-degenerate (capped) periodic orbit.  The
\emph{Conley--Zehnder index} $\MUCZ(\bar{x})\in\Z$ is defined, up to a
sign, as in \cite{Sa,SZ}. More specifically, in this paper, the
Conley--Zehnder index is the negative of that in \cite{Sa}. In other
words, we normalize $\MUCZ$ so that $\MUCZ(\bar{x})=n$ when $x$ is a
non-degenerate maximum (with trivial capping) of an autonomous
Hamiltonian with small Hessian. The \emph{mean index}
$\Delta_H(\bx)\in\R$ measures, roughly speaking, the total angle swept
by certain eigenvalues with absolute value one of the linearized flow
$d\varphi^t_H$ along $x$ with respect to the trivialization associated
with the capping; see \cite{Lo,SZ}. (Sometimes we use the notation
$\Delta(\bx)$ when the Hamiltonian is clear from the context.) The
mean index is defined regardless of whether $x$ is degenerate or not,
and $\Delta_H(\bx)$ depends continuously on $H$ and $\bx$ in the
obvious sense. When $x$ is non-degenerate, we have
\begin{equation}
\label{eq:Delta-MUCZ}
0<|\Delta_H(\bx)-\MUCZ(H,\bx)|<n.
\end{equation}
Furthermore, the mean index is also homogeneous with respect to iteration:
\begin{equation}
\label{eq:index-hom}
\Delta_{H^{\nat \ka}}(\bx^\ka)=\ka\Delta_H(\bx).
\end{equation}

We let $\CA_H(x)\in S^1_{\lambda_0}=\R/\lambda_0\Z$ be the action
$\CA_H(\bx)$ taken modulo $\lambda_0$, and hence independent of the
capping.  (The mean index $\Delta_H(x)\in S^1_{2N}=\R/2N\Z$ can be
defined in a similar fashion.)  Finally, recall that the
\emph{augmented action} of a one-periodic (or $\ka$-periodic) orbit
$x$ is the difference
\begin{equation}
\label{eq:aug-action}
\tA_H(x)=\CA_H(\bx)-\frac{\lambda}{2}\Delta_H(\bx),
\end{equation}
where on the right we use an arbitrary capping of $x$; see
\cite{GG:gaps}. By \eqref{eq:recap} and \eqref{eq:recap2}, the augmented
action is independent of the capping and, by \eqref{eq:action-hom} and
\eqref{eq:index-hom}, homogeneous:
\begin{equation}
\label{eq:aug-action-hom}
\tA_{H^{\nat \ka}}(x^\ka)=\ka\tA_H(x).
\end{equation}

\subsection{Floer and quantum (co)homology} 

\subsubsection{Floer homology}
\label{sec:Floer}
In this subsection, we very briefly recall, mainly to set notation,
the construction of the filtered Floer homology for strictly monotone
symplectic manifolds, i.e., monotone manifolds with $\lambda\neq 0$
and $c_1(TM)\mid_{\pi_2(M)}\neq 0$. We refer the reader to, e.g.,
\cite{Fl,HS,MS,Sa,SZ} and for detailed accounts of the construction
and for additional references.

Throughout this paper, all complexes and homology groups are, for the sake of
simplicity, defined over the ground field $\F=\Z_2$. Let $H$ be a one-periodic
non-degenerate Hamiltonian on $M$.  Denote by $\CF^{(-\infty,\,
  b)}_k(H)$, where $b\in (-\infty,\,\infty]$ is not in $\CS(H)$, the
vector space of finite formal sums
$$ 
\alpha=\sum_{\bar{x}\in \bPP(H)} \alpha_{\bar{x}}\bar{x}. 
$$
Here $\alpha_{\bar{x}}\in\F$ and $|\bx|:=\MUCZ(\bar{x})+n=k$ and
$\CA_H(\bar{x})<b$. (Since $M$ is strictly monotone there is no need
to consider semi-infinite sums.)  The graded $\F$-vector space
$\CF^{(-\infty,\, b)}_*(H)$ is endowed with the Floer differential
$\p$ counting the anti-gradient trajectories of the action functional.

More specifically, $\p$ is defined as follows. Fix a
generic (one-periodic in time) almost complex structure $J$ compatible
with $\omega$ and consider solutions $u\colon \R\times S^1\to M$ of
the Floer equation
\begin{equation}
\label{eq:Fl}
\p_s u+J\p_t u =-\nabla H .
\end{equation}
Here the gradient on the right is taken with respect to the
one-periodic in time metric
$\left<\cdot\,,\cdot\right>=\omega(\cdot\,,J\cdot)$ on $M$, and $(s,t)$
are the coordinates on $\R\times S^1$. Denote by $\CM(\bx,\by)$ the
space of solutions of the Floer equation, \eqref{eq:Fl}, asymptotic to
$\bx$ at $-\infty$ and $\by$ at $\infty$ and such that the capping of
$\by$ is equivalent to the capping obtained by attaching $u$ to the
capping of $\bx$. The space $\CM(\bx,\by)$ carries a natural free
$\R$-action by shifts along the $s$-axis. Set
$\widehat{\CM}(\bx,\by)=\CM(\bx,\by)/\R$.  For a generic choice of $J$, we
have $\dim \widehat{\CM}(\bx,\by) =|\bx|-|\by|-1$. Let $m(\bx,\by)\in \F$
stand for the
parity of $\#\widehat{\CM}(\bx,\by)$ when $\widehat{\CM}(\bx,\by)$ is
zero-dimensional (one can show that then it is finite) and zero
otherwise. Finally, set
$$
\p \bx=\sum_{\by} m(\bx,\by)\by.
$$
It is well known that $\p^2=0$; see the references above.

The complexes $\CF^{(-\infty,\, b)}_*(H)$ equipped with the
differential $\p$ form a filtration of the
total Floer complex $\CF_*(H):=\CF^{(-\infty,\, \infty)}_*(H)$. We set
$\CF^{(c,\, c')}_*(H):=\CF^{(-\infty,\,
  c')}_*(H)/\CF^{(-\infty,\,c)}_*(H)$, where  $c$ and $c'$ are not in
$\CS(H)$ and $-\infty\leq
c<c'\leq\infty$. The homology groups of
these complexes are called the \emph{filtered Floer homology} of $H$
and denoted by $\HF^{(c,\, c')}_*(H)$ or by $\HF_*(H)$ when
$(c,\,c')=(-\infty,\,\infty)$. These groups are independent of the
choice of $J$. Every $\F$-vector space $\CF_k(H)$ is
finite-dimensional since $M$ is strictly monotone.

Recall also that for any $u\in \CM(\bx,\by)$ we have
$$
\CA_H(\bx)-\CA_H(\by)=E(u),
\textrm{ where }
E(u):=\int_{\R\times S^1} \|\p u\|^2\, dsdt,
$$
and that $E(u)$ is referred to as the \emph{energy} of $u$.

The total Floer complex and homology are modules over the
\emph{Novikov ring} $\Lambda$ whose action on the complex is
essentially by recapping the orbits. To define this ring, let us first
set some notation.

 Let $\omega(A)$ and $\left<c_1(TM),A\right>$ denote the
integrals of $\omega$ and, respectively, $c_1(TM)$ over a cycle $A$.
Set
$$
I_\omega(A)=-\omega(A)\text{ and } I_{c_1}(A)=-2\left<c_1(TM),
  A\right>,
$$
where $A\in\pi_2(M)$.  Thus $I_\omega=(\lambda/2)I_{c_1}$ since $M$ is
assumed to be monotone.  Let $\Gamma=\pi_2(M)/{\ker I_\omega}
=\pi_2(M)/\ker I_{c_1}$.  In other words, $\Gamma$ is the quotient of
$\pi_2(M)$ by the equivalence relation where two spheres $A$ and $A'$
are considered to be equivalent if $\left<c_1(TM),
  A\right>=\left<c_1(TM), A'\right>$ or, equivalently,
$\omega(A)=\omega(A')$. Clearly $\Gamma\simeq\Z$, and the
homomorphisms $I_\omega$ and $I_{c_1}$ descend to $\Gamma$ from
$\pi_2(M)$.

The group $\Gamma$ acts on $\CF_*(H)$ and on $\HF_*(H)$ via recapping:
an element $A\in \Gamma$ acts on a capped one-periodic orbit $\bar{x}$
of $H$ by attaching the sphere $A$ to the original capping. Denoting
the resulting capped orbit by $\bx\# A$, we have
\begin{equation}
\label{eq:recap}
\MUCZ(\bx\# A)=\MUCZ(\bx)+ I_{c_1}(A)
\text{ and }
\CA_H(\bx\# A)=\CA_H(\bx)+I_\omega(A).
\end{equation}
In a similar vein, 
\begin{equation}
\label{eq:recap2}
\Delta_H(\bx\# A)=\Delta_H(\bx)+ I_{c_1}(A)
\end{equation}
regardless of whether $x$ is non-degenerate or not.

Since $M$ is strictly monotone, the Novikov ring $\Lambda$ may simply
be taken to be the group algebra $\F[\Gamma]$ of $\Gamma$ over $\F$,
i.e., the $\F$-algebra of formal finite linear combinations $\sum
\alpha_A e^A$, where $\alpha_A\in\F$. The Novikov ring $\Lambda$ is
graded by setting $|e^A|=I_{c_1}(A)$ for $A\in\Gamma$.  The action of
$\Gamma$ turns $\CF_*(H)$ and $\HF_*(H)$ (but not $\CF^{(c,\,
  c')}_*(H)$ or $\HF^{(c,\, c')}_*(H)$)
 into
$\Lambda$-modules.  Recall that $\Gamma\simeq \Z$ and denote by $A_0$
the generator of $\Gamma$ with $I_{c_1}(A_0)=-2N$. Setting
$q=e^{A_0}\in \Lambda$, we have $|q|=-2N$, and the Novikov ring
$\Lambda$ is thus the ring of Laurent polynomials $\F[q^{-1},q]$.

The definition of Floer homology extends to all, not necessarily
non-degenerate, Hamiltonians by continuity.  Let $H$ be an arbitrary
(one-periodic in time) Hamiltonian on $M$ and let the end points $c$
and $c'$ of the action interval be outside $\CS(H)$. We
set
$$
\HF^{(c,\, c')}_*(H)=\HF^{(c,\, c')}_*(\tH),
$$
where $\tH$ is a non-degenerate, small perturbation of $H$. It is well
known that the right hand side is independent of $\tH$ as long as the
latter is sufficiently close to $H$.  Working with filtered Floer
homology, \emph{we always assume that the end points of the
  action interval are not in the action spectrum.} Finally, note that
everywhere in this discussion we can replace one-periodic Hamiltonians
by $\ka$-periodic.

The total Floer homology is independent of the Hamiltonian and 
isomorphic to the homology of $M$. More precisely, we have
$$
\HF_*(H)\cong \H_ {*}(M)\otimes \Lambda
$$
as graded $\Lambda$-modules.

\subsubsection{Quantum homology}
\label{sec:QH}
The total Floer homology $\HF_*(H)$, equipped with the pair-of-pants
product, is an algebra over the Novikov ring $\Lambda$. This algebra
is isomorphic to the (small) \emph{quantum homology} $\HQ_*(M)$ of $M$; see,
e.g., \cite{MS}. On the level of $\Lambda$-modules, we have
\begin{equation}
\label{eq:floer-qh}
\HQ_*(M)=\H_*(M)\otimes \Lambda
\end{equation}
with tensor product grading. Thus $|\alpha\otimes e^A|=
|\alpha|+I_{c_1}(A)$, where $\alpha\in \H_*(M)$ and $A\in\Gamma$.  The
isomorphism between $\HF_*(H)$ and $\HQ_*(M)$ is usually defined by using the
PSS-homomorphism; see \cite{PSS} or \cite{MS,U:product}.  Alternatively, at
least in the monotone case, it
can be obtained from a homotopy of $H$ to an autonomous $C^2$-small
Hamiltonian, see \cite{Fl} and, e.g., \cite{HS,Ono:AC}.

The \emph{quantum product} $\alpha*\beta$ of two ordinary homology
classes $\alpha$ and $\beta$ is defined~as
\begin{equation}
\label{eq:qp}
\alpha*\beta=\sum_{A\in\Gamma} (\alpha*\beta)_A \,e^{A},
\end{equation}
where the class $(\alpha*\beta)_A\in \H_*(M)$ is determined via
certain Gromov--Witten invariants of $M$ and has degree
$|\alpha|+|\beta|-2n-I_{c_1}(A)$; see \cite{MS}. Thus
$|\alpha*\beta|=|\alpha|+|\beta|-2n$.

Note that $(\alpha*\beta)_0=\alpha\cap \beta$, where $\cap$ stands for
the cap product and $\alpha$ and $\beta$ are ordinary homology
classes.  Recall also that in \eqref{eq:qp} it suffices to limit the
summation to the negative cone $I_\omega(A)\leq 0$. In particular,
under our assumptions on $M$, we can write
$$
\alpha*\beta=\alpha\cap \beta+\sum_{k>0} (\alpha*\beta)_k \, q^{k},
$$
where $|(\alpha*\beta)_k|=|\alpha|+|\beta|-2n + 2Nk$.  This sum is finite.

The product $*$ is a $\Lambda$-linear, associative, graded-commutative
product on $\HQ_*(M)$. The fundamental class $[M]$ is the unit in the
algebra $\HQ_*(M)$. Thus $a\alpha=(a[M])*\alpha$, where $a\in\Lambda$
and $\alpha\in \H_*(M)$, and $|a\alpha|=|a|+|\alpha|$. By the very
definition, the ordinary homology $\H_*(M)$ is canonically embedded in
$\HQ_*(M)$.

The map $I_\omega$ extends to $\HQ_*(M)$ as
$$
I_\omega(\alpha) =\max\,\{I_\omega(A)\mid \alpha_A\neq 0\}
=\max\,\{-\lambda_0 k \mid \alpha_k\neq 0\},
$$
where $\alpha=\sum \alpha_A e^A=\sum \alpha_k q^k$. We have
\begin{equation}
\label{eq:shift}
I_\omega(\alpha+\beta)\leq \max\,\{ I_\omega(\alpha),I_\omega(\beta)\}
\end{equation}
and 
\begin{equation}
\label{eq:shift2}
I_\omega(\alpha*\beta)\leq I_\omega(\alpha)+I_\omega(\beta).
\end{equation}

\begin{Example} 
\label{ex:cpn}
Let $M=\CP^n$. Then $N=n+1$ and $\HQ_*(\CP^n)$ is the quotient of
$\F[u]\otimes\Lambda$, where $u$ is the generator of
$\H_{2n-2}(\CP^n)$, by the ideal generated by the relation
$u^{n+1}=q[M]$.  Thus $u^k=u\cap \ldots\cap u$ ($k$ times) when
$0\leq k\leq n$, and $[\pt]*u=q[M]$. For this and further examples of
quantum homology calculations and relevant references, we refer the
reader to, e.g., \cite{MS}.
\end{Example}

\subsection{Cap product}
\label{sec:cap-product}
In what follows, it is convenient to view the product on the Floer or
quantum homology from a slightly different angle -- namely, as a
``module structure'' over $\HQ_*(M)$ on the collection of filtered
Floer homology groups. We refer to this structure, somewhat
misleadingly, as the \emph{cap product} even though no cohomology is
explicitly in the picture.

Let us now describe the action of the quantum homology on the filtered
Floer homology in more detail. Let $H$ be a non-degenerate Hamiltonian
and $J$ be a generic almost complex structure. Let $[\sigma]\in
\H_*(M)$. Pick a generic cycle $\sigma$ representing $[\sigma]$ and
denote by $\CM(\bx,\by;\sigma)$ the moduli space of solutions of the
Floer equation, \eqref{eq:Fl}, asymptotic to $\bx$ at $-\infty$ and
$\by$ at $\infty$ and such that $u(0,0)\in \sigma$. For a generic
choice of $\sigma$, we have $\dim \CM(\bx,\by;\sigma)
=|\bx|-|\by|-\codim(\sigma)$. Let $m(\bx,\by;\sigma)\in \F$ be the parity of
$\#\CM(\bx,\by;\sigma)$ when this moduli space is zero-dimensional
(one can show that then it is finite) and zero otherwise. Set
$$
\Phi_{\sigma}(\bx)=\sum_{\by} m(\bx,\by;\sigma)\by.
$$
For any $c$ and $c'$ outside $\CS(H)$, the map
$$
\Phi_\sigma \colon \CF^{(c,\,c')}_*(H)\to \CF^{(c,\,c')}_{*-\codim(\sigma)}(H)
$$
commutes with the Floer differential and descends to a map
$$
\Phi_{[\sigma]} \colon \HF^{(c,\,c')}_*(H)\to \HF^{(c,\,c')}_{*-\codim(\sigma)}(H),
$$
which is independent of the choice of the cycle $\sigma$ representing
$[\sigma]$. The analytical details of this construction and complete
proofs can be found in, e.g., \cite{LO}, in much greater generality
than is needed here. Clearly,
\begin{equation}
\label{eq:id}
\Phi_{[M]}=\id.
\end{equation}

The action of the class $\alpha=q^\nu[\sigma]\in \HQ_*(M)$ is induced by the map
$$
\Phi_{q^\nu\sigma}(\bx):=\sum_{\by} m(q^\nu\bx,\by;\sigma)\by.
$$
Here, as in Section \ref{sec:Floer}, $q=e^{A_0}$ where $A_0$ is the generator of
$\Gamma$ with $I_{c_1}(A_0)=-2N$. It is routine to
check that 
$\Phi_{q^\nu[\sigma]}=q^\nu\Phi_{[\sigma]}$. (This is a consequence of
the fact that $ \CM(q^\nu\bx,\by;\sigma)=\CM(\bx,q^{-\nu}\by;\sigma)$.)
Note that now the action
interval is shifted by $I_\omega(\alpha)$, i.e.,
\begin{equation}
\label{eq:cap-action}
\Phi_\alpha \colon \HF^{(c,\,c')}_*(H)\to 
\HF^{(c,\,c')+I_\omega(\alpha)}_{*-2n+|\alpha|}(H),
\end{equation}
where $(c,\,c')+s$ stands for $(c+s,\,c'+s)$.

By linearity over $\Lambda$, we extend $\Phi_\alpha$ to all $\alpha\in
\HQ_*(M)$ so that \eqref{eq:cap-action} still holds. The maps
$\Phi_\alpha$ are linear in $\alpha$ once the shift in filtration is
taken into account. Namely, let $(a,\,b)$ be any interval such that
$a\geq c+ \max\,\{ I_\omega(\alpha),I_\omega(\beta)\}$ and $b\geq c'+
\max\,\{ I_\omega(\alpha),I_\omega(\beta)\}$. Then the
quotient--inclusion maps $\iota_\alpha$ and $\iota_\beta$ from
$\HF^{(c,\,c')+I_\omega(\alpha)}_*(H)$ and, respectively,
$\HF^{(c,\,c')+I_\omega(\beta)}_*(H)$ to $H^{(a,\,b)}_*(H)$ are defined;
see, e.g., \cite[Example 3.3]{Gi:Co}.  Clearly, $a\geq
c+I_\omega(\alpha+\beta)$ and $b\geq c'+I_\omega(\alpha+\beta)$ by
\eqref{eq:shift}, and we also have the map
$\iota_{\alpha+\beta}$ from the target space of $\Phi_{\alpha+\beta}$
to $\HF^{(a,\,b)}_*(H)$. Additivity takes the form
$$
\iota_{\alpha+\beta}\Phi_{\alpha+\beta}=\iota_\alpha\Phi_{\alpha}+\iota_\beta\Phi_{\beta}.
$$

The maps $\Phi_{\alpha}$, for all $\alpha\in\HQ_*(M)$, fit together
to form an action of the quantum homology on the collection of
filtered Floer homology groups. 

This action is also multiplicative. In other words, we have
\begin{equation}
\label{eq:action}
\Phi_\beta\Phi_\alpha=\Phi_{\beta*\alpha}.
\end{equation}
Note that here, as in the case of additivity, the maps on the two
sides of the identity have, in general, different target spaces. Thus
\eqref{eq:action} should, more accurately, be understood as that for
any interval $(a,\,b)$ with $a\geq c+I_\omega(\alpha)+I_\omega(\beta)$
and $b\geq c'+I_\omega(\alpha)+I_\omega(\beta)$ the following diagram
commutes:
\begin{equation}
\label{eq:ass}
{\xymatrix{
\HF^{(c,\,c')}_*(H) \ar[r]^-{\Phi_\alpha}
\ar[dr]_-{\Phi_{\beta*\alpha}}
&\HF^{(c,\,c')+I_\omega(\alpha)}_{*-2n+|\alpha|}(H)
\ar[r]^-{\Phi_\beta}
&\HF^{(c,\,c')+I_\omega(\alpha)+I_\omega(\beta)}_{*-4n+|\alpha|+|\beta|}(H)
\ar[d]
\\
&\HF^{(c,\,c')+I_\omega(\beta*\alpha)}_{*-2n+|\beta*\alpha|}(H)
\ar[r]
&\HF^{(a,\,b)}_{*-2n+|\beta*\alpha|}(H)
}}
\end{equation}
Here the vertical arrow and the bottom horizontal arrow are again the
natural quotient--inclusion maps whose existence is guaranteed by
our choice of $a$ and $b$ and \eqref{eq:shift2}. Note that
\eqref{eq:action} can be thought of as a form of associativity of the
product in quantum or Floer homology; \eqref{eq:action} was
essentially established in \cite{LO} and \cite{PSS}; see also
\cite[Remark 12.3.3]{MS}.

On the total Floer homology $\HF_*(H)\cong \HQ_*(M)$, the cap product
coincides with the quantum or pair-of-pants product.

\subsection{Non-contractible orbits}
\label{sec:non-contr}
To deal with the case of non-contractible orbits, we need to strengthen
the monotonicity requirement on $M$. Namely, we say that $M$ is
\emph{toroidally monotone} if for every map $v\colon \T^2\to M$ we
have $\left<[\omega], [v]\right>=\lambda\left<c_1(TM),[v]\right>$ for
some constant $\lambda\geq 0$ independent of $v$. This condition is in
general stronger than monotonicity, but weaker than requiring that
$[\omega]=\lambda c_1(TM)$. (This can be easily seen by examining
surfaces or products of surfaces.) Furthermore, assume from now on
that $M$ is in addition strictly monotone. Then the toroidal
monotonicity constant $\lambda$ is equal to the ordinary monotonicity
constant and the minimal toroidal Chern number and the toroidal
rationality constant (both defined in the obvious way) agree with
their spherical counterparts.

Let $\zeta$ be a free homotopy class of maps $S^1\to M$. Fix a
\emph{reference loop} $z\in\zeta$ and a symplectic trivialization of $TM$
along $z$. (In fact, it would be sufficient to fix a trivialization of
the canonical bundle of $M$ along $z$.) A capped loop $x$ is a loop in
$\zeta$ together with a cylinder (i.e., a homotopy) connecting it to
$z$. Two cappings are considered equivalent when
$\left<c_1(TM),[v]\right>=0$, where $v\colon \T^2\to M$ is the torus
obtained by attaching the cappings to each other. (Due to the toroidal
monotonicity condition, we also have $\left<[\omega],[v]\right>=0$.) As in the
contractible case, we denote the capping by a bar.

Let now $H$ be a Hamiltonian on $M$. We consider capped one-periodic
(or $\ka$-periodic) orbits of $H$ in the class $\zeta$. For such orbits,
the action $\CA_H(\bx)$ and the mean index $\Delta_H(\bx)$ (and the
Conley--Zehnder index $\MUCZ(\bx)$ when $x$ is non-degenerate) are
obviously well defined. Likewise, the augmented action
\eqref{eq:aug-action} is still defined and independent of the capping,
and \eqref{eq:Delta-MUCZ} holds.  

The construction of the filtered Floer complex $\CF_*^{(c,\,
  c')}(H,\zeta)$ and the filtered Floer homology $\HF_*^{(c,\,
  c')}(H,\zeta)$ goes through exactly as in the contractible case.

The Floer complex and homology are again modules over a Novikov ring
$\Lambda'$. In general, this ring can be different from the spherical
one. This is the case, for instance, when $M$ is symplectically
aspherical and toroidally monotone (e.g., $M=\T^2$) and $\Lambda=\F$
while $\Lambda'=\F[q,q^{-1}]$. However, when $M$ is both strictly
monotone and toroidally monotone, as is assumed here, the natural
inclusion $\Lambda\to\Lambda'$ is an isomorphism and we can keep the
notation $\Lambda$ for both of the Novikov rings.

Since all one-periodic orbits of a $C^2$-small autonomous Hamiltonian
are constant, and hence contractible, the total non-contractible Floer
homology is trivial: $\HF_*(H,\zeta)=0$ whenever $\zeta\neq
0$. Furthermore, the pair-of-paints product is not defined on
non-contractible filtered Floer homology for an individual class
$\zeta$. (Such a product between $\HF_*^{(c_1,\, c'_1)}(H,\zeta_1)$
and $\HF_*^{(c_2,\, c'_2)}(K,\zeta_2)$ would take values in
$\HF_*^{(c_1+c_2,\, c'_1+c'_2)}(H\nat K,\zeta_1+\zeta_2)$.)
However, the cap product with ordinary $\HQ_*(M)$ is still defined and
the discussion from Section \ref{sec:cap-product} carries over
word-for-word to the non-contractible case.

The action and the index of periodic orbits (and hence the grading and
filtration of the Floer complex) do depend on the choice of $z$ and,
in the index case, on the trivialization of $TM$ along $z$.  Whenever
we consider the iteration $H^{\nat \ka}$ of $H$, we simultaneously
iterate the class $\zeta$ and the reference curve $z$ (i.e., pass to
$\ka\cdot \zeta$ and $z^\ka$) and, in the obvious sense, also iterate
the trivialization. Under these assumptions, the action, the mean
index, and the augmented action are homogeneous with respect to the
iterations, i.e., \eqref{eq:action-hom}, \eqref{eq:index-hom}, and
\eqref{eq:aug-action-hom} remain valid.

\begin{Remark}
\label{rmk:0vs-non0}
The index and action conventions of this section can be used even when 
$\zeta=0$, resulting only in a shift of the standard grading
and filtration in the Floer homology.
\end{Remark}

\section{Hyperbolic fixed points}
\label{sec:hyperbolic}

Our goal in this section is to establish a technical result underpinning
the proof of Theorem \ref{thm:main}. This is the fact that the
energy required for a Floer connecting trajectory to approach an
iterated hyperbolic orbit and cross its fixed neighborhood is
bounded away from zero by a constant independent of the order of iteration.

\subsection{Setting}
\label{sec:setting}
Let us state precisely the conditions needed for this energy
lower bound to hold. Let $H$ be a one-periodic in time Hamiltonian on a
symplectic manifold $(M,\omega)$, which is not required to
be monotone or closed or satisfy any extra requirements at all; for
the result we are interested in is essentially local. Fix an
almost complex structure $J$, again one-periodic in time,
compatible with $\omega$.  

We consider solutions $u\colon \Sigma\to M$ of the Floer equation
\eqref{eq:Fl}, where $\Sigma\subset \R\times S^1_\ka$ is a closed
domain (i.e., a closed subset with non-empty interior). Now, however,
in contrast with Section \ref{sec:Floer}, the period $\ka$ is not
necessarily fixed, and the domain $\Sigma$ of $u$ need not be the
entire cylinder $\R\times S^1_\ka$.  By definition, the energy of $u$
is
$$
E(u)=\int_\Sigma \|\p_s u\|^2 \, ds dt.
$$
Here $\|\cdot\|$ stands for the norm with respect to
$\left<\cdot\,,\cdot\right>=\omega(\cdot,J\cdot)$, and hence
$\|\cdot\|$ also depends on $J$.

Let $\gamma$ be a hyperbolic (not necessarily contractible)
one-periodic orbit of $H$. We say that $u$ is asymptotic to
$\gamma^\ka$, the $\ka$th iteration of $\gamma$, as $s\to\infty$ if
$\Sigma$ contains some cylinder $[s_0,\infty)\times S^1_\ka$ and
$u(s,t)\to \gamma^\ka(t)$ $C^\infty$-uniformly in $t$ as
$s\to\infty$.
  
Finally, let $U$ be a fixed (sufficiently small) closed neighborhood
of $\gamma$ with smooth boundary or, more precisely, such a
neighborhood of the natural lift of $\gamma$ to $S^1\times M$.

\begin{Theorem}[Ball--crossing Energy Theorem]
\label{thm:energy}
There exists a constant $c_\infty>0$, independent of $\ka$ and
$\Sigma$,
such that for any solution $u$ of the Floer equation, \eqref{eq:Fl},
with $u(\p \Sigma)\subset \p U$ and $\p \Sigma\neq\emptyset$, 
which is asymptotic to $\gamma^\ka$ as $s\to\infty$, we have
\begin{equation}
\label{eq:energy}
E(u)>c_\infty .
\end{equation}
Moreover, the constant $c_\infty$ can be chosen to make
\eqref{eq:energy} hold for all $\ka$-periodic almost complex
structures (with varying $\ka$) $C^\infty$-close to $J$ uniformly on
$\R\times U$.
\end{Theorem}

\begin{Remark} 
\label{rmk:energy}
The most important point of this theorem is the fact that the energy
required to approach $\gamma^\ka$ through $U$ is bounded away from
zero by a constant $c_\infty$ independent of the iteration $\ka$ and
also of the domain $\Sigma$ of $u$. (Naturally, $c_\infty$ depends on
$H$ and $J$ and $\omega$.) A similar lower bound obviously holds for
solutions of the Floer equation leaving $\gamma^\ka$ through $U$,
i.e., asymptotic to $\gamma^\ka$ as $s\to -\infty$. Clearly, the
requirement that the orbit $\gamma$ is one-periodic can be replaced by
the assumption that it is just a periodic orbit.  Furthermore, in the
``moreover'' part of the theorem, we can take, for instance,
$\ka$-periodic almost complex structures (with $\ka$ fixed)
$C^\infty$-close uniformly on $U$ to the $\ka$-periodic extension of
$J$.  Periodic perturbations of $H$ can be incorporated into the
theorem in a similar fashion.
\end{Remark}

\begin{Remark} 
\label{rmk:energy2}
 It might also be worth pointing out that the condition
  $u(\p \Sigma)\subset \p U$ should be understood as that the map
  $(s,t)\mapsto (t, u(s,t))$ sends $\p \Sigma$ to the boundary of the
  domain in $S^1_\ka\times M$ covering $U\subset S^1\times
  M$. Furthermore, note that, when dealing with a countable collection
  of maps $u$, we can always ensure that the domains $\Sigma$ have
  smooth boundary by slightly shrinking $U$.
\end{Remark}

We apply Theorem \ref{thm:energy} in the following setting. Let $u$ be
a solution of the Floer equation for $H^{\nat \kappa}$, which is
asymptotic to $\gamma^\kappa$ on one side and to some
$\kappa$-periodic orbit $x$ on the other. Assume furthermore that $x$
does not enter a neighborhood $U$ of $\gamma$. Then $E(u)>c_\infty$
for some constant $c_\infty>0$, which, by the theorem, is independent
of $\kappa$ and $u$. Moreover, we can replace $H^{\nat \kappa}$ by any
$\kappa$-periodic Hamiltonian $K$ equal to $H^{\nat \kappa}$ on $U$.

\begin{Remark}
\label{rmk:elliptic} 
The assumption that $\gamma$ is hyperbolic is absolutely crucial in
Theorem \ref{thm:energy}. Consider, for instance, the 
Hamiltonian $H(z)=a |z|^2$ generating an irrational rotation of
$\C=\R^2$. Then, as a direct calculation shows, the ball crossing
energy can get arbitrary small for arbitrarily large iterations
$\ka$. Alternatively, this can be seen by examining the height
function $H$ on $S^2$ and observing that, for a suitable choice of
$\ka$ and of the cappings of the polls $x$ and $y$, the difference
$|\CA_{H^{\nat \ka}}(\bx^\kappa)-\CA_{H^{\nat \ka}}(\by^\kappa)|$ can
be arbitrarily small while there are always Floer trajectories
connecting $\bx^\kappa$ and $\by^\kappa$. (The latter fact follows
from the structure of the cap or quantum product on $S^2$; see
Example \ref{ex:cpn}.)
\end{Remark}

\subsection{Proof of Theorem \ref{thm:energy}} The key feature of
hyperbolic orbits the argument relies on is that, given a closed
neighborhood of $\gamma$, the orbit of $\varphi^t_H$ with the initial
condition on (or near) the boundary of the neighborhood cannot stay
within the neighborhood simultaneously for large positive and large
negative times.

To simplify the setting, we first observe that without loss of
generality we can assume that $\gamma(t)$ is a constant orbit, which
we still denote by $\gamma$. This is a consequence of the fact that
there exists a one-periodic loop of Hamiltonian diffeomorphisms
$\eta_t$, defined on a neighborhood of $\gamma$, such that
$\eta_t(\gamma(0))=\gamma(t)$ (see, e.g., \cite[Section
5.1]{Gi:CC}), which allows us to replace $\varphi_H^t$ by
$\eta_t^{-1}\varphi_H^t$ and $H$ by the corresponding Hamiltonian. 
This step is not really necessary, but it does simplify the notation
and the geometrical picture, and, since the loop $\eta_t$ is in
general only local, does not require the orbit $\gamma$ to be
contractible. From now on, we can assume that $U$ and other
neighborhoods of $\gamma$ are actually subsets of $M$ rather than of
$S^1_\ka\times M$. 

Next, let us fix another closed neighborhood $B\subset
\mathrm{int}(U)$ of $\gamma$ with smooth boundary. If $B$ is
sufficiently small, there exists a constant $L_0>0$, depending only on
$B$ and $H$, such that for all initial times $\tau\in [0,\,1]$ no
integral curve of $H$ passing through a point of $\p B$ (or near $\p
B$) at the moment $\tau$ can stay in $B$ for all $t$ with
$|t-\tau|<L_0$.  This readily follows from the assumptions that
$\gamma$ is hyperbolic and $H$ is one-periodic in time. (Note that the
role of $\tau$ here is to account for the fact that the vector field
$X_{H}$ and the flow $\varphi_H^t$ are time-dependent. If
$\varphi_H^t$ were a true flow, we would be able to take $\tau=0$. In
our setting, where $H$ is one-periodic in time, it suffices to require
that $\tau\in [0,\,1]$.)

The idea of the proof is that if there is a sequence of solutions $u_i$ of the Floer
equation with $E(u_i)\to 0$, this sequence converges to a zero
energy solution defined on a domain in $\C$. The limit solution
is independent of $s$, and hence an integral curve of $H$. The
sequence $u_i$ can be chosen so that this integral curve is
defined on the interval $[\tau-L,\, \tau+L]$ for some $L>L_0$ and $\tau\in
[0,\,1]$, contained in $B$, and tangent to $\p B$ at the moment
$\tau$. This is impossible since $\gamma$ is hyperbolic.

Thus, arguing by contradiction, assume now that there exist a sequence
of $\ka_i$-periodic almost complex structures $J_i$ on $M$, compatible
with $\omega$ and $C^\infty$-converging to $J$ uniformly on $\R\times
U$, and a sequence $u_i\colon\Sigma_i\to U$ of solutions of
\eqref{eq:Fl} (for $J_i$ with $\ka=\ka_i$), satisfying the hypotheses
of Theorem \ref{thm:energy} and such that $E(u_i)\to 0$.  As we show
below, we may assume without loss of generality that the maps $u_i$
and the almost complex structures $J_i$ have the following properties:
\begin{itemize}

\item[(i)] The domains $\Sigma_i$ have smooth boundary.

\item[(ii)] The region $[0,\infty)\times S^1_{\ka_i}$ is the largest half-cylinder in
$\Sigma_i$ mapped by $u_i$ into $B$, i.e.,
$u_i\big([0,\infty)\times S^1_{\ka_i}\big)\subset B$ and
$u_i\big(\{0\}\times S^1_{\ka_i}\big)$ touches $\p B$ at at least one
point $u_i(0,\tau_i)$, and furthermore $0\leq \tau_i\leq 1$.

\item[(iii)] The sequences $\tau_i$ and
$u_i(0,\tau_i)$ converge.

\end{itemize}

Here (i) can be ensured by slightly altering the domain $U$; see
Remark \ref{rmk:energy2}. Furthermore, the first part of (ii) readily
follows since $H$ is independent of $s$, and hence \eqref{eq:Fl} is
translation invariant.  (This requires changing $u_i$ by applying a
translation in $s$, which clearly does not change the energy.) Next,
to ensure that $0\leq \tau_i\leq 1$, we apply an integer translation
in $t$ to $u_i$ and $J_i$. Since the almost complex structures $J_i$
$C^\infty$-converge to $J$ uniformly in $t\in\R$, the same is true for
the translated almost complex structures. This procedure changes the
almost complex structures $J_i$ and the solutions $u_i$, but it does not
effect the energy of $u_i$, which therefore still goes to zero.
Finally, it suffices to pass to a subsequence to establish (iii). 

Let us lift the domains $\Sigma_i$ to the domains $\hat{\Sigma}_i$ in
the universal covering $\C$ of $\R\times S^1_{\ka_i}$ and view the
maps $u_i$
(keeping the notation) as maps $u_i\colon \hat{\Sigma}_i\to M$, which
are $\ka_i$-periodic in $t$. As is well--known, the graph $\hat{u}_i$
of $u_i$ is a $\tilde{J}_i$-holomorphic curve in $\C\times M$ with
respect to an almost complex structure $\tilde{J}_i$ which
incorporates both $J_i$ and $X_H$. Recall that the projection
$\pi\colon \C\times M\to \C$ is holomorphic, and hence the projection
of $\hat{u}_i$ to $\C$ is also a holomorphic map. Set
$\tau=\lim\tau_i\in [0,\,1]$.

Pick arbitrary constants $L>L_0$ and $a>0$ and fix a rectangle 
$$
\Pi=[-a,a]\times [\tau- L,\tau+ L]\subset \C.
$$
From now on we will focus on the restrictions $u_i\mid_\Pi$, which we
still denote by $u_i$. Let $\tilde{u}_i$ be the graph of this restriction,
i.e., intersection of $\hat{u}_i$ with $\Pi\times U$. Clearly the
boundary of $\tilde{u}_i$ lies in $\p (\Pi\times U)$ and 
$$
\Area(\tilde{u}_i)\leq\Area(\Pi)+E(u_i)<\const,
$$ 
where the constant on the right is independent of $i$. Let us now
shrink $\Pi$ and $U$ slightly. To be more precise, set
$$
\Pi'=[-a',a']\times [\tau- L',\tau+ L']\subset \Pi,
$$
where $0<a'<a$ and $L_0<L'<L$, and let $U'$ be a closed neighborhood
of $\gamma$ such that $B\subset \mathrm{int}(U')$ and $U'\subset
\mathrm{int}(U)$.

By Fish's version of the Gromov compactness theorem, \cite[Theorem
A]{Fi}, the intersections of the $\tilde{J}_i$-holomorphic curves
$\tilde{u}_i$ with $\Pi'\times U'$ Gromov--converge, after passing if
necessary to a subsequence, to a (cusp) $\tilde{J}$-holomorphic curve
$\tilde{u}$ in $\Pi'\times U'$ with boundary in $\p (\Pi'\times
U')$. This holomorphic curve is a union of multi-sections over subsets
of $\Pi'$ and possibly some components contained in the fibers of
$\pi$ (the bubbles) with boundary in $\p U'$. The latter are points
since $E(u_i)\to 0$. Furthermore, $\tilde{u}$ is in fact a unique
section over some subset of $\Pi'$. The reason is that the
intersection index of $\tilde{u}$ with the fiber over a regular point
$(s,t)$ of its projection to $\Pi'$ is either one or zero -- the
intersection index of $\tilde{u}_i$ with the fiber. For instance, the
index is one when $(s,t)$ is in the domain of each $u_i$ and the
distance from $u_i(s,t)$ to $\p U'$ stays bounded away from zero as
$i\to \infty$.  (Here and below we loosely follow the proof of
\cite[Lemma 2.3]{McL}. Note also that we need the parameters $a'$ and
$L'$ and the neighborhood $U'$ to be ``generic''.)

To summarize, $\tilde{u}$ is the graph of a solution $u$ of the Floer
equation, \eqref{eq:Fl}, defined on some (obviously connected) subset
$D$ of $\Pi'$. Moreover, as is easy to see from \cite{Fi}, after
making an arbitrarily small change to $a'$ and $L'$ we can ensure that
the domain $D$ of $u$ has piece-wise smooth boundary. The maps $u_i$
uniformly converge to $u$ on compact subsets of $\mathrm{int}(D)$.

Observe now that $D$ contains the half-rectangle $\Pi^+=
\{s>0\}\cap\Pi'$. This is a consequence of the fact that each
$\tilde{u}_i$ projects surjectively onto the $\{s\geq 0\}$-part of
$\Pi$ or, in other words, this part of $\Pi$ is in the domain of
$u_i$. As a consequence, $D$ also contains the closure of $\Pi^+$ and,
in particular, the point $(0,\tau)$. Moreover, this point is in
the interior of $D$ since the distance from the points $u_i(0,\tau)\in
B$ to $\p U'$ stays bounded away from zero. Thus 
$$
u(0,\tau)=p:=\lim u_i(0,\tau_i)\in \p B.
$$

Clearly, the solution $u$ has zero energy. Thus $\p_s u(s,t)=0$
identically on $D$, and hence $u(s,t)$ is an integral curve $u(t)$ of
the Hamiltonian flow of $H$. This integral curve
passes through the point $p\in \p B$ at the moment
$\tau$, and $u(t)\in B$ for all $t\in [\tau-L',\tau+L']$, which
is impossible due to our choice of $L_0$ and the fact that
$L'>L_0$. This contradiction completes the proof of the theorem.

\begin{Remark} 
\label{rmk:growth}
  Under some additional assumptions on $J$ and $H$ and $u$ in Theorem
  \ref{thm:energy}, one can obtain a much more precise lower bound on the
  energy of $u$. Assume, for instance, that $H$ is a quadratic
  hyperbolic Hamiltonian on $\R^{2n}$, and hence $\varphi_H^t$ is a
  linear flow. Furthermore, let us require that $\Sigma=[0,\infty)\times
  S^1_\ka$ and $u$ be a solution of \eqref{eq:Fl} for a linear
  complex structure $J$ suitably adapted to $H$. Then
  $$
  E(u)>\const \|u(0)\|^2_{L^2},
  $$
  where $u(0)=u(0,\cdot)$ and $\const>0$ is independent of $\ka$ and $u$,
  but depends on $H$ and $J$. (This can be proved by analyzing
  \eqref{eq:Fl} via the Fourier expansion of $u$ in $t$.) In
  particular, if $u(\partial \Sigma)$ lies outside the ball of radius
  $R$, i.e., $\|u\|_{L^\infty}>R$, we have $E(u)>\const\cdot R^2
  \ka$. As a consequence, under these conditions, the ball--crossing
  energy $c_\infty$ grows linearly in $\ka$.
\end{Remark}

\begin{Remark}
  An argument similar to (and in fact simpler than) the proof of
  Theorem \ref{thm:energy} shows that, for a fixed period $\kappa$,
  the ball crossing energy is bounded away from zero by a constant
  independent of a solution of the Floer equation for any isolated
  $\kappa$-periodic orbit of an arbitrary $\kappa$-periodic
  Hamiltonian; cf.\ the proof of \cite[Lemma 2.3]{McL}.  However, it
  is essential that in this case the period $\kappa$ is fixed. The
  bound depends on the neighborhood of the orbit, the Hamiltonian and
  the almost complex structure, and, in contrast with Theorem
  \ref{thm:energy}, on $\kappa$.
\end{Remark}

\begin{Remark}
  As is clear from the proof, Theorem \ref{thm:energy} holds in some
  instances for periodic orbits which are not necessarily hyperbolic;
  for instance, the argument applies to $\gamma\equiv 0$ for the
  degenerate Hamiltonians $x^2-y^4$ and $x^4-y^4$ and the ``monkey
  saddle'' on $\R^2$.
\end{Remark}

\section{Proof of the main theorem}
\label{sec:proof-relations}

\subsection{Idea of the proof}

Fix a neighborhood $U$ of $\gamma$ as in Theorem \ref{thm:energy}.  To
explain the idea of the proof, let us, focusing on Case (i), first
show that $c_\infty\leq \lambda_0$ whenever no other periodic orbit
enters $U$.  Assume the contrary: $c_\infty>\lambda_0$. Without loss
of generality, we may also assume that the orbit $\gamma$ is
one-periodic and that $\CA_H(\bg)=0$ for some capping of
$\gamma$. Then the chain $\bg \in\CF_*^{(a,\,b)}(H)$ is closed for any
interval $(a,\,b)$ containing $[-\lambda_0,0]$ and such that
$b-a<c_\infty$.  Moreover, $[\bg]\neq 0$ in $\HF_*^{(a,\,b)}(H)$.
Acting on $[\bg]$ by $\alpha$ and applying \eqref{eq:action} with
$\beta*\alpha=q[M]$, we see that $\Phi_\alpha([\bg])\neq 0$ in
$\HF_*^{(-\lambda_0,\,0)}(H)$.  Thus $\bg$ is connected by a Floer
trajectory $u$ to some orbit $\by$ with action in the range
$(-\lambda_0,\,0)$. This is impossible since
$E(u)=|\CA_H(\by)|<c_\infty$.

Now, if we knew that, as in Remark \ref{rmk:growth}, the energy
$c_\infty$ goes to infinity as $\kappa$ grows, we could make
$c_\infty$ greater than $\lambda_0$ by passing to an iteration.  This
would prove that there are periodic orbits entering an arbitrarily
small neighborhood of $\gamma$ -- an assertion much stronger than the
theorem. Of course, we do not know weather or not $c_\infty\to\infty$
as $\kappa\to\infty$. In the proof, we circumvent this difficulty by
replacing $H$ with a carefully chosen iteration $H^{\nat \kappa}$
(using the condition that $N\geq n/2+1$) so that the above argument
still goes through even when $c_\infty<\lambda_0$; see Lemma
\ref{lemma:closed}.

\subsection{Proof of Theorem \ref{thm:main}}

Arguing by contradiction, let us assume that $\varphi_H$ has finitely
many periodic orbits. In this case, by replacing $H$ with its
iteration, we can also assume that $\gamma$ is a one-periodic orbit
and that $H$ is perfect in the terminology of \cite{GK}: all periodic
orbits of $\varphi_H$ are fixed points. Furthermore, again by passing
if necessary to an iteration, we can guarantee that the mean index of
$\gamma$ (equal to the Conley--Zehnder of $\gamma$) with respect to
any capping is divisible by $2N$. Then there exists a capping such
that the mean index is zero. Let $\bg$ be the orbit $\gamma$ equipped
with this capping: $\Delta(\bg)=0$.  We keep the notation $H$ for this
iterated Hamiltonian and assume it to be one-periodic in time, which
can always be achieved by a reparametrization. Finally, by adding a
constant to $H$, we can ensure that $\CA_H(\bg)=0$.

Fix a one-periodic in time almost complex structure $J^0$. Let $U$ be
a small closed neighborhood of $\gamma$ such
that no periodic orbit of $H$ other than $\gamma$ intersects $U$. By
Theorem \ref{thm:energy} applied to $U$, there exists a constant
$c_\infty>0$ such that, for any $\ka$, every non-trivial
$\ka$-periodic solution of the Floer equation for the pair $(H,J^0)$
asymptotic to $\gamma^\ka$ as $s\to\pm\infty$ has energy greater than~$c_\infty$. 

Denote by $a_i\in S^1_{\lambda_0}$ the actions of one-periodic orbits
of $H$ taken up to $\lambda_0$, and hence independent of
capping.  Let also $\ta_i\in\R$ stand for the
augmented actions of one-periodic orbits. Due to the assumption
that $H$ has finitely many fixed points, the collections
$\{a_i\}$ and $\{\ta_i\}$ are finite. 

Furthermore, fix a large constant
$C>0$ and a small constant $\eps>0$. The values of these constants are to be
specified later; see \eqref{eq:C} and \eqref{eq:eps-delta}.

As is easy to show using the Kronecker theorem and
\eqref{eq:aug-action-hom}, there exists an integer period $\ka >0$ such
that for all $i$
\begin{equation}
\label{eq:T}
\| \ka\cdot a_i\|_{\lambda_0}<\eps
\textrm{ and either }\ta_i=0\textrm{ or
}|\ka\cdot \ta_i|>C.
\end{equation}
Here $\|a\|_{\lambda_0}\in [0,\lambda_0/2]$ stands for the distance
from $a\in S^1_{\lambda_0}$ to $0$.

Let $K$ be a $\ka$-periodic Hamiltonian sufficiently $C^2$-close to
$H^{\nat \ka}$. Denote by $\PP$ the collection of contractible
$\ka$-periodic orbits of $K$ and by $\bar{\PP}$ the collection of
capped $\ka$-periodic orbits. Then \eqref{eq:T} readily implies that
for every $x\in\PP$ we have
\begin{gather}
\label{eq:K1}
\|\CA_K(x)\|_{\lambda_0}<\eps\textrm{ and }\\ 
\label{eq:K2}
\textrm{ either }|\tA_K(x)|<\delta
\textrm{ or
}|\tA_K(x)|>C,
\end{gather}
where $\delta>0$, to be specified later (see \eqref{eq:eps-delta}),
can be made arbitrarily small. In particular, by \eqref{eq:K1},
\emph{$\CS(K)$ is contained in the $\eps$-neighborhood of
  $\lambda_0\Z$}.

On the other hand, assume that $K$ is non-degenerate (and again
$\ka$-periodic) and equal to $H^{\nat \ka}$ on $U$. (We do not need
$K$ to be $C^2$-close to $H^{\nat \ka}$.) Then for any $\ka$-periodic
almost complex structure $J$, which is sufficiently close to (the
$\ka$-periodic extension of) $J^0$, all non-trivial $\ka$-periodic
solutions of the Floer equation for the pair $(K,J)$ asymptotic to
$\gamma^\ka$ as $s\to\pm\infty$ have energy greater than $c_\infty$.
This follows from the ``moreover'' part of Theorem \ref{thm:energy}
(see Remark \ref{rmk:energy}) or, alternatively, can be easily
established as a consequence of the compactness theorem for solutions
of the Floer equation for a fixed period, see, e.g., \cite[Corollary
3.4]{Sa}.

Let us now fix a $\ka$-periodic Hamiltonian $K$ meeting all of the
above conditions: $K$ is non-degenerate, sufficiently $C^2$-close to
$H^{\nat \ka}$, and equal to $H^{\nat \ka}$ on $U$. (When $H^{\nat
  \kappa}$ is non-degenerate, we can take $K=H^{\nat \kappa}$.)
Furthermore, fix a $\ka$-periodic almost complex structure $J$ such that
the regularity requirements for the pair $(K,J)$ are satisfied, and
every non-trivial solution of the Floer equation for $(K,J)$
asymptotic to $\gamma^\ka$ as $s\to\pm\infty$ has energy greater than
$c_\infty$. In particular, as a consequence of the regularity, the
filtered Floer complex $\CF_*(K)$ for the pair $(K,J)$ is defined.

Before proceeding with the proof, let us spell out, in the logical
order, the sequence of choices made above. Once $H$ has been made
perfect, we start by fixing $U$ and an
almost complex structure $J^0$. These choices determine 
$c_\infty$. Next we choose a large constant $C>0$ and small constants
$\eps>0$ and $\delta>0$; we will explicitly specify our requirements
on these constants shortly. Then we pick $\ka$ to satisfy
\eqref{eq:T}, and then the Hamiltonian $K$ meeting several conditions
including \eqref{eq:K1} and \eqref{eq:K2}. Finally, we fix the almost
complex structure $J$.

In the rest of the proof, the exact restrictions on the constants $C$,
$\eps$ and $\delta$ are essential. The constant $C$ is chosen so that
\begin{equation}
\label{eq:C}
C>2n\lambda+\nu\lambda_0.
\end{equation}
The constants $\eps$ and $\delta$ are positive, and 
\begin{equation}
\label{eq:eps-delta}
\eps<c_\infty\textrm{ and } 2(\eps+\delta)<\lambda.
\end{equation}
Note that $\kappa$ depends on $\eps$ (and of course $C$), which in
turn depends on the choice of $c_\infty$ and ultimately on the
choice of $U$.

From now on, we work exclusively with the Hamiltonian $K$, its
$\ka$-periodic orbits, and the Floer equation for $(K,J)$.  To simplify
the notation, we write $\gamma$ and $\bg$ in place of $\gamma^\ka$
and $\bg^\ka$. Our ultimate goal is to show that there exists an orbit
$\by$ of $K$ with action not in the $\eps$-neighborhood of
$\lambda_0\Z$, which is impossible by \eqref{eq:K1}.  This will
complete the proof of the theorem.

\begin{Lemma}
\label{lemma:closed}
Let $C'=C-\lambda(n+1)/2$. The orbit $\bg$ is not connected by a
solution of the Floer equation of relative index $\pm 1$ to any
$\bx\in \bPP$ with action in $(-C',\,C')$. In particular,
$\bg$ is closed in $\CF_*^{(-C',\,C')}(K)$ and $[\bg]\neq 0$ in
$\HF_*^{(-C',\,C')}(K)$.
Moreover, $\bg$ must enter every cycle representing $[\bg]$ in
$\HF_*^{(-C',\,C')}(K)$.
\end{Lemma}

\begin{proof} Arguing by contradiction, assume that $\bx$ is connected
  to $\bg$ by a solution $u$ of the Floer equation of relative index
  $\pm 1$.  Thus $\MUCZ(\bx)=\pm 1$, and hence $|\Delta_K(\bx)|< n+1$
  by \eqref{eq:Delta-MUCZ} and also $x\neq \gamma$. By \eqref{eq:K2},
  either $|\tA_K(x)|<\delta$ or $|\tA_K(x)|>C$.

  Let us first consider the former case: $|\tA_K(x)|<\delta$. We have
  $E(u)>c_\infty>\eps$, and therefore 
$$
|\CA_K(\bx)|=E(u)>\lambda_0-\eps,
$$
by \eqref{eq:K1} and \eqref{eq:eps-delta}. It follows from the
condition $|\tA_K(x)|<\delta$ and the second inequality in
\eqref{eq:eps-delta} that
$$
|\Delta_K(\bx)|>\frac{2}{\lambda}(\lambda_0-\eps-\delta)=2N-
\frac{2(\eps+\delta)}{\lambda}>2N-1.
$$
Thus, using the requirement that $N\geq n/2+1$, we have
$$
|\MUCZ(\bx)|> 2N-1-n\geq n+2-1-n=1,
$$
which is impossible since $\MUCZ(\bx)=\pm 1$.

In the latter case, $|\tA_K(x)|>C$, we have
$$
|\CA_K(\bx)|>C-\frac{\lambda}{2}|\Delta_K(\bx)|>C-\frac{\lambda}{2}(n+1)=C',
$$
where we used \eqref{eq:Delta-MUCZ} in the last inequality. Hence the
orbit $\bx$ is outside  the action range $(-C',\, C')$.
\end{proof}

Since the Floer differential commutes with recapping, the same is true
for the orbit $q^\nu\bg=\bg\#(\nu A_0)$ with the shifted action range 
$(-C',\,C')-\nu\lambda_0$.

Let now $(a,\,b)$ be any open interval containing
$[-\nu\lambda_0,\,0]$ and contained in the intersection of the
intervals $(-C',\,C')$ and $(-C',\,C')-\nu\lambda_0$. Such an interval
exists since $-C'<-\lambda_0$ and $C'-\nu\lambda_0>0$ due to our
choice of $C$; see \eqref{eq:C}. Then the above assertions hold for
both capped orbits $\bg$ and $q^\nu\bg$ and the interval $(a,\,b)$:
these orbits are not connected by a Floer trajectory of relative index
$\pm 1$ to any capped orbit with action in this interval; the chains
$\bg$ and $q^\nu\bg$ are closed in $\CF_*^{(a,\,b)}(K)$; the capped
orbits must enter any representatives of the classes $[\bg]$ and,
respectively, $q^\nu[\bg]$ in $\HF_*^{(a,\,b)}(K)$; and, in
particular, these classes are both non-zero.

\begin{Lemma}
\label{lemma:orbits}
The Hamiltonian $K$ has a periodic orbit with action outside the
$\eps$-neighborhood of $\lambda_0\Z$.
\end{Lemma}

The main theorem follows from this lemma. For the existence of such an
orbit contradicts \eqref{eq:K1}.

\begin{proof}
  Applying \eqref{eq:ass} and \eqref{eq:action}  to $[\bg]$ with $(c,\, c')=(a,\,
  b)$, we see that the image (under the quotient-inclusion map) in
  $\HF_*^{(a,\,b)}(K)$ of the class $\Phi_\alpha ([\bg])$ is non-zero;
  for $q^\nu[\bg]\neq 0$ in this homology group.

  Recall that $\alpha$ and $\beta$ are both
  ordinary homology classes and let, as in Section \ref{sec:cap-product},
  $\sigma$ and $\eta$ be generic cycles representing 
  $\alpha$ and, respectively, $\beta$. Then the chain
  $\Phi_\sigma(\bg)$ in $\CF_*^{(a,\,b)}(K)$ is also non-zero.
  Moreover, the chain $\Phi_\eta\Phi_\sigma(\bg)$ represents the class $q^\nu[\bg]$,
  and hence the orbit $q^\nu\bg$ enters this chain. As a consequence, there
  exists an orbit $\by$ in the chain $\Phi_\sigma(\bg)$ which is
  connected to both $\bg$ and $q^\nu\bg$ by Floer trajectories. 

Therefore,
\begin{equation}
\label{eq:case-ii}
-\eps> \CA_K(\by)>-\nu\lambda_0+\eps .
\end{equation}

When $\nu=1$, \eqref{eq:case-ii} turns into 
$-\eps> \CA_K(\by)>-\lambda_0+\eps$, which
concludes the proof of the lemma in Case (i).

To establish the lemma in Case (ii), we first observe that
$\MUCZ(\by)=|\alpha|-2n$, and thus $|\Delta_K(\by)|<4n$. (In fact,
$-3n<\Delta_K(\by)<n-1$.)  Next, as in the proof of Lemma
\ref{lemma:closed}, we consider two sub-cases: $|\tA_K(\by)|<\delta$
and $|\tA_K(\by)|>C$.

In the former case, we argue by contradiction. Assume that
$\CA_K(\by)$ is in the $\eps$-neighborhood of $\lambda_0\Z$. Then, by
\eqref{eq:case-ii}, $\CA_K(\by)<-\lambda_0+\eps$, and hence we obtain
using \eqref{eq:eps-delta} that
$$
\Delta_K(\by)<\frac{2}{\lambda}\big(-\lambda_0+\eps+\delta\big)
=-2N+\frac{2(\eps+\delta)}{\lambda}<-2N+1.
$$
Therefore,
$$
|\alpha|-2n=\MUCZ(\by)<-2N+1+n
$$
or, in other words,
$$
|\alpha|<3n+1-2N,
$$
which contradicts the assumption of the theorem that $|\alpha|\geq
3n+1-2N$. 

In the latter case where $|\tA_K(\by)|>C$, we have
$$
|\CA_K(\by)|>C-\frac{\lambda}{2}|\Delta_K(\by)|>C-2n\lambda>\nu\lambda_0
$$
by \eqref{eq:C}, which is impossible due to \eqref{eq:case-ii}.
\end{proof}

This concludes the proof of the theorem.

\begin{Remark}
\label{rmk:refinement}
It is not hard to see from the proof of Theorem \ref{thm:main} that in
Case (i) the condition that $\alpha$ and $\beta$ are ordinary homology
classes can be slightly relaxed and replaced, for instance, by the
requirement that $I_\omega(\alpha)=0$ and $I_\omega(\beta)\geq
-\lambda_0$.
\end{Remark}

\begin{Remark}[Proof in the case of non-contractible orbits]
\label{rmk:non-contr-pf}
When the orbit $\gamma$ is non-contractible, the above argument goes
through essentially word-for-word, requiring only minimal modifications
(in fact, simplifications) in the beginning of the proof. Indeed, let
$\zeta$ be the free homotopy class of $\gamma$.  Without loss of
generality, we may assume that $\ka\zeta\neq 0$ for all positive integers
$\ka$ or, alternatively, use Remark \ref{rmk:0vs-non0}. By passing if
necessary to the second iteration, we can also ensure that
$d\varphi_H$ at $T_{\gamma(0)}M$ has an even number of real eigenvalues in
$(-1,\,0)$. Then there exists a trivialization of $TM$ along $\gamma$
such that $\Delta_H(\gamma)=0=\MUCZ(\gamma)$. We take $z=\gamma$ with
this trivialization as a reference loop. In the notation and
conventions of Section \ref{sec:non-contr}, the rest of the proof
remains unchanged.
\end{Remark}


\begin{thebibliography}{CKRTZ}

\bibitem[BH]{BH}
B. Bramham, H. Hofer, First steps towards a symplectic dynamics,
Preprint 2011, arXiv:1102.3723.

\bibitem[CGG]{CGG} M. Chance, V.L. Ginzburg, B.Z. G\"urel,
  Action-index relations for perfect Hamiltonian diffeomorphisms,
  Preprint 2011, arXiv:1110.6728; to appear in \emph{J. Sympl.\ Geom.}

\bibitem[CKRTZ]{CKRTZ}
B. Collier, E. Kerman, B. Reiniger, B. Turmunkh, A. Zimmer
A symplectic proof of a theorem of Franks, Preprint 2011,
arXiv:1107.1282; to appear in \emph{Compos.\ Math.} 

\bibitem[Fi]{Fi}
J.W. Fish, Target-local Gromov compactness,
\emph{Geom.\ Topol.}, \textbf{15} (2011), 765--826.

\bibitem[Fl]{Fl}
A. Floer,
Symplectic fixed points and holomorphic spheres, \emph{Comm.\ Math.\
Phys.}, \textbf{120} (1989), 575--611.

\bibitem[Fr1]{Fr1}
J. Franks, Geodesics on $S^2$ and periodic points of annulus
homeomorphisms, \emph{Invent.\ Math.}, \textbf{108} (1992), 403--418. 

\bibitem[Fr2]{Fr2} J. Franks, Area preserving homeomorphisms of open
  surfaces of genus zero, \emph{New York Jour.\ of Math.}, \textbf{2} (1996),
    1--19.

\bibitem[FH]{FH}
J. Franks, M. Handel,
Periodic points of Hamiltonian surface diffeomorphisms, \emph{Geom.\
  Topol.}, \textbf{7} (2003), 713--756.

\bibitem[Gi1]{Gi:Co}
V.L. Ginzburg,
Coisotropic intersections, \emph{Duke Math.\ J.}, \textbf{140} (2007),
111--163.

\bibitem[Gi2]{Gi:CC}
V.L. Ginzburg,
The Conley conjecture, \emph{Ann.\ of Math.} (2), \textbf{172} (2010), 1127--1180.

\bibitem[GG1]{GG:gaps}
V.L. Ginzburg, B.Z. G\"urel,
Action and index spectra and periodic orbits in Hamiltonian dynamics,
\emph{Geom.\ Topol.}, \textbf{13} (2009), 2745--2805.

\bibitem[GG2]{GG:generic}
V.L. Ginzburg, B.Z. G\"urel,
On the generic existence of periodic orbits in Hamiltonian dynamics,
\emph{J. Mod.\ Dyn.}, \textbf{4}
(2009), 595--610. 

\bibitem[GG3]{GG:neg-mon} 
V.L. Ginzburg, B.Z. G\"urel,
Conley conjecture for negative monotone symplectic manifolds,
\emph{Int.\ Math.\ Res.\ Not.\ IMRN}, 2011, doi:10.1093/imrn/rnr081.

 
\bibitem[GK]{GK}
V.L. Ginzburg, E. Kerman, 
Homological resonances for Hamiltonian diffeomorphisms and Reeb flows,
\emph{Int.\ Math.\ Res.\ Not.\ IMRN}, \textbf{2010}, 53--68. 

\bibitem[G\"u]{Gu}
B.Z. G\"urel, Periodic orbits of Hamiltonian systems linear and
hyperbolic at infinity, in preparation.

\bibitem[Ha]{Ha} 
S. Hayashi, Connecting invariant manifolds and the
  solution of the $C^1$-stability and $\Omega$-stability conjectures for
  flows, \emph{Ann.\ of Math.} (2), \textbf{145} (1997), 81--137.

\bibitem[He]{He:irr}
D. Hein,
The Conley conjecture for irrational symplectic manifolds, Preprint
2009, arXiv:0912.2064; to appear in \emph{J. Sympl.\ Geom.}

\bibitem[Hi]{Hi}
N. Hingston,
Subharmonic solutions of Hamiltonian equations on tori, 
 \emph{Ann.\ of Math.} (2), \textbf{170} (2009), 525--560.

\bibitem[HS]{HS}
H. Hofer, D. Salamon, 
Floer homology and Novikov rings, in \emph{The Floer Memorial Volume},
483--524, Progr.\ Math., 133, Birkh\"auser, Basel, 1995.

\bibitem[HZ]{HZ}
H. Hofer, E. Zehnder,
\emph{Symplectic Invariants and Hamiltonian Dynamics}, Birk\"auser,
1994.

\bibitem[Ke]{Ke:JMD}
E. Kerman,
On primes and period growth for Hamiltonian diffeomorphisms,
\emph{J. Mod.\ Dyn.}, \textbf{6} (2012), 41--58.

\bibitem[LO]{LO}
H.V. L\^{e}, K. Ono, 
Cup-length estimates for symplectic fixed points, in \emph{Contact and
Symplectic Geometry (Cambridge, 1994)}, 268--295, Publ.\ Newton Inst.,
8, Cambridge Univ.\ Press, Cambridge, 1996.

\bibitem[LeC]{LeC}
P. Le Calvez,
Periodic orbits of Hamiltonian homeomorphisms of surfaces,
\emph{Duke Math.\ J.}, \textbf{133} (2006),  125--184.

\bibitem[Lo]{Lo}
Y. Long,  \emph{Index Theory for Symplectic Paths with Applications}, 
Progress in Mathematics, 207. Birkh\"auser Verlag, Basel, 2002. 

\bibitem[MS]{MS}
D. McDuff, D. Salamon,
\emph{J-holomorphic Curves and Symplectic Topology}, Colloquium
publications, vol.\ 52, AMS, Providence, RI, 2004.

\bibitem[McL]{McL}
M. McLean, 	
Local Floer homology and infinitely many simple Reeb orbits, Preprint 2012,
arXiv:1202.0528.

\bibitem[On]{Ono:AC}
K. Ono, 
On the Arnold conjecture for weakly monotone symplectic manifolds,
\emph{Invent.\ Math.}, \textbf{119} (1995), 519--537. 

\bibitem[PSS]{PSS}
S. Piunikhin, D. Salamon, M. Schwarz,
Symplectic Floer--Donaldson theory and quantum cohomology, in
\emph{Contact and Symplectic Geometry} (Cambridge, 1994), 171--200;
C.B. Thomas (Ed.), Publ.\ Newton Inst., 8, Cambridge Univ.\ Press,
Cambridge, 1996.

\bibitem[PR]{PR}
C. Pugh, C. Robinson, The $C^1$ closing lemma, including Hamiltonians,
\emph{Ergodic Theory Dynam.\ Systems}, \textbf{3} (1983), 261--313.

\bibitem[Ru]{Ru} 
D. Ruelle, \emph{Elements of Differentiable Dynamics
    and Bifurcation Theory}, Academic Press, Inc., Boston, MA, 1989.

\bibitem[Sa]{Sa}
D.A. Salamon,
Lectures on Floer homology, in \emph{Symplectic Geometry and
Topology}, Eds: Y. Eliashberg and L. Traynor, IAS/Park City
Mathematics series, \textbf{7} (1999), pp.\ 143--230.

\bibitem[SZ]{SZ}
D. Salamon, E. Zehnder,
Morse theory for periodic solutions of Hamiltonian systems and the
Maslov index, \emph{Comm.\ Pure Appl.\ Math.}, \textbf{45} (1992),
1303--1360.

\bibitem[Us]{U:product}
M. Usher, 
Deformed Hamiltonian Floer theory, capacity estimates, and Calabi
quasimorphisms, \emph{Geom.\ Topol.}, \textbf{15} (2011), 1313--1417.

\bibitem[Xi]{Xi}
Z. Xia, 
Homoclinic points in symplectic and volume-preserving diffeomorphisms,
\emph{Comm.\ Math.\ Phys.}, \textbf{177} (1996), 435--449. 

\bibitem[Ze]{Ze} E. Zehnder, \emph{Lectures on Dynamical
    Systems. Hamiltonian Vector Fields and Symplectic Capacities} EMS
    Textbooks in Mathematics. European Mathematical Society (EMS),
    Z\"urich, 2010.

\end{thebibliography}
\end{document}